\documentclass{article}
\usepackage[utf8]{inputenc}

\usepackage{quiver}
\RequirePackage{amsmath}
\RequirePackage{amsthm}
\RequirePackage{amssymb}

\usepackage{dsfont}

\title{Delta operation at height 2}
\date{}
\author{Noam Nissan}

\renewcommand{\tilde}[1]{\widetilde{#1}}

\newtheorem{theorem}{Theorem}[subsection]
\newtheorem{example}[theorem]{Example}
\newtheorem{corollary}[theorem]{Corollary}
\newtheorem{proposition}[theorem]{Proposition}
\newtheorem{definition}[theorem]{Definition}
\newtheorem{lemma}[theorem]{Lemma}
\newtheorem{remark}[theorem]{Remark}
\newtheorem{notation}[theorem]{Notation}

\newcommand\explicitset[1]{\left\{#1\right\}}
\newcommand\cardinality[1]{\left|#1\right|}

\newcommand\parenth[1]{\left(#1\right)}

\newcommand\squarebrack[1]{\left[#1\right]}

\DeclareMathOperator{\spf}{spf}
\DeclareMathOperator{\spec}{spec}

\newcommand{\spaces}{\mathcal{S}}
\newcommand\catname[1]{\textbf{#1}}
\newcommand\cat[1]{\mathcal{#1}}

\DeclareMathOperator{\nm}{Nm}
\DeclareMathOperator{\tr}{tr}
\DeclareMathOperator{\colim}{colim}
\DeclareMathOperator{\id}{id}
\DeclareMathOperator{\fun}{Fun}
\DeclareMathOperator{\map}{Map}
\DeclareMathOperator{\eoperad}{\mathbb{E}}

\newcommand{\calg}{\catname{CAlg}}
\DeclareMathOperator{\fin}{fin}
\DeclareMathOperator{\cmon}{CMon}
\DeclareMathOperator{\spaan}{Span}
\DeclareMathOperator{\en}{End}
\DeclareMathOperator{\iso}{Iso}

\DeclareMathOperator{\im}{Im}

\newcommand{\integers}{\mathbb{Z}}
\newcommand{\rationals}{\mathbb{Q}}
\newcommand{\finitefield}{\mathbb{F}}
\newcommand{\spectra}{\catname{Sp}}

\newcommand{\smunit}{\mathds{1}}
\newcommand{\complexprojective}{\mathbb{CP}^\infty}
\newcommand{\suspension}{\Sigma}
\newcommand\tmnt[1]{\catname{Mod}_{E_#1}(\spectra_{K(#1)})}
\DeclareMathOperator{\cl}{Cl}
\DeclareMathOperator{\Sum}{Sum}
\DeclareMathOperator{\stab}{Stab}

\begin{document}

\maketitle

Submitted as part of the requirements of obtaining the degree of A Master of Science under the supervision of Prof. Tomer Moshe Schlank in the Hebrew University in March 2022. The author's name durning the preparation of this work was Noam Zimhoni.

\begin{abstract}
    We compute the delta power operation for morava E-theory of height 2 at the prime 3. The delta power operation was defined using the notion of higher semi additivity by Shachar Carmeli, Tomer M. Schlank and Lior Yanovski. We briefly survey the basic definitions in higher semi-additivity. Using explicit formulas for moduli spaces of elliptic curves and computations done by Y. Zhu for the total power operation we provide an explicit formula for the delta power operation. We obtain the numerical result that the delta operation of a polynomial is a polynomial which can be explained by a certain connection which was described by M. Hopkins, C. Rezk and later N. Stapleton between similar operations and the Hecke action on the functions on the moduli space of elliptic curves.
\end{abstract}

\section{Introduction}
 
In recent years, the notion of higher semiadditivity has become central in chromatic homotopy theory. The first nontrivial example of a higher semiadditive $\infty$-category is the category $\spectra_{K(n)}$ of $K(n)$-local spectra, as shown in \cite{HL}.

Let $k$ be a field of characteristic $p$ (we will consider the case $k = \finitefield_q$), and let $G$ be a formal group of height $n$ over $k$. There are two important homotopy commutative spectra associated to this data, denoted $K(n)$ and $E_n$. The ring $\pi_0(E_n)$ is local, and its residue field is canonically isomorphic to $\pi_0(K(n)) = \finitefield_q$. Both spectra are periodic. Explicitly,
\[
\pi_*(E_n) = \pi_0(E_n)[u^{\pm 1}]
\quad\text{and}\quad
\pi_*(K(n)) = \pi_0(K(n))[v_n^{\pm 1}].
\]
Multiplication by $u$ and $v_n$ induces equivalences
\[
E_n \cong \suspension^2 E_n
\quad\text{and}\quad
K(n) \cong \suspension^{2(p^n-1)} K(n),
\]
respectively.

This spectra are complex oriented. The most important feature of these spectra is their associated Quillen formal group. We have
\[
\pi_*K(n)^ {\complexprojective} = \pi_*(K(n))[[t]] = \finitefield_q[v_n^{\pm 1}][[t]]
\quad\text{and}\quad
\pi_*E_n^{\complexprojective} = \pi_*(E_n)[[t]].
\]
Both carry Hopf algebra structures induced by the group structure on $\complexprojective$. It follows that
\[
\spf \pi_0K(n)^{\complexprojective}
\quad\text{and}\quad
\spf \pi_0E_n^{\complexprojective}
\]
are formal groups. The formal group $\spf K(n)^* \complexprojective$ is the original formal group, while $\spf E_n^* \complexprojective$ is its universal deformation. The $K(n)$-local category $\spectra_{K(n)}$ is the full subcategory of spectra spanned by objects that receive no nonzero maps from $K(n)$-acyclic spectra.

Higher semiadditivity yields a notion of integration over $\pi$-finite spaces for $K(n)$-local spectra. Moreover, the $K(n)$-local category is $\infty$-semiadditively symmetric monoidal. This implies a distributivity relation between multiplication and integration on $K(n)$-local ring spectra. This distributivity follows from the fact that the (completed) tensor product on $\spectra_{K(n)}$ preserves colimits indexed by $\pi$-finite spaces, as well as all other colimits.

In \cite{CSY1}, Carmeli, Schlank, and Yanovski define a power operation for commutative ring spectra in $\infty$-semiadditively symmetric monoidal $\infty$-categories. They use this operation to show that the $T(n)$-local category is also $\infty$-semiadditively symmetric monoidal. In Section~\ref{sec:ambyanddelta}, we give a brief introduction to higher semiadditivity, focusing on definitions, and we recall the definition of the $\delta$ operation.


The goal of this project is to compute the $\delta$ operation on the ring spectrum $E_2$ in the category $\tmnt{2}$ of $K(2)$-local $E_2$-modules. We work at the prime $3$, which is the smallest prime for which $\Sigma_p \neq C_p$. When $C_p = \Sigma_p$, some of our methods degenerate, since the $C_p$ power operation coincides with the usual total power operation. We expect the methods developed here to work at any prime.

In Sections~\ref{sec:powerop} and~\ref{sec:charoftoalpower}, we recall the previous results on which our calculation relies.


The first main result we use is Zhu's computation in \cite{ZHU} of the additive total power operation at height $2$. Zhu gives an explicit formula for a ring map
\[
\psi^3 \colon \pi_0(E) \to \pi_0(E^{B\Sigma_p})/T,
\]
where $T$ is the transfer ideal. In Section~\ref{sec:powerop}, we recall this result and the related work of Strickland and Rezk on which it depends. 

The second tool is a special case of the computation of Barthel and Stapleton in \cite{BS} of the character of the power operation. This work computes the rationalization of the total power operation
\[
\widetilde{\psi^3} \colon \pi_0(E) \to \pi_0(E^{B\Sigma_p}).
\]
In Section~\ref{sec:charoftoalpower}, we recall this result and briefly survey HKR character theory.

In the final section, Section~\ref{sec:computation}, we assemble the results described above and complete the computation.

\subsection{Statement of the main result}

Let $\finitefield_9 = \finitefield_3[c]/(c^2+1)$. Let $C$ be the elliptic curve over $\finitefield_9$ defined by the equation
\[
y^2 + xy - y = x^3 - x^2.
\]
The curve $C$ is supersingular. Hence its formal completion $\widehat{C_e}$ is a formal group of height $2$. We take the associated Morava $E$-theory $E$.

\begin{proposition}\cite[Section~2.1]{ZHU}\label{prop:compofE0}
The ring $\pi_0(E)$ is the universal deformation ring of $C$. There is an isomorphism
\[
\pi_0(E) \cong  \integers_3[[h]][c]/(c^2+1-h).,
\]
$h$ corresponds to the Hesse invariant, whose vanishing characterizes when the elliptic curve is supersingular.
\end{proposition}

\begin{remark}
    Modulo the maximal ideal $(3,h)$, the element $c \in \pi_0(E)$ agrees with $c \in \finitefield_9$.
\end{remark}

The main theorem of this work is the following.

\begin{theorem}(Theorem~\ref{thm:computationofdelta})
Let $B \in M_8(\pi_0(E))$ be the matrix defined in Notation~\ref{notation:ugly_stuff}. Let $x \in \pi_0(E) =\integers_3[[h]][c]/(c^2+1-h)$, regarded as a formal power series in the variable $h$, and let $x(B)$ denote the matrix-valued series obtained by substituting $h = B$ in the power series $x$. This series converges in the $(3,h)$-adic topology to a matrix $x(B) \in M_8(\pi_0(E))$. The operation $\delta$ is given by
\[
\delta(x) = 3x - \frac{x^3 + \tr x(B)}{3}.
\]
\end{theorem}

\begin{remark}
Although the formula involves elements of $\pi_0(E)$ that are not in $\integers[h]$, such as $c = \sqrt{h-1}$ and $\frac{1}{h-17}$ appearing in the coefficients of $B$, the trace of $B$ and of all its powers is polynomial in $h$ with integer coefficients. For example, $\delta(h)$ is the polynomial
\[
-h^{3} + 18h^{2} - 119h + 102.
\]

This phenomenon suggests that $\delta$ should admit an integral definition, and possibly a geometric interpretation. We do not pursue this question here. It is indeed explained in \cite{STAPLETON} that some closely reltated operation agrees with the Hecke operator $T_p$.
\end{remark}

\subsubsection*{Acknowledgements}
This work was produced while the author was a student in the Hebrew university of Jerusalem, Israel under the supervision of Prof. Tomer Moshe Schlanck.

Special thanks to the Seminarak research group in Jerusalem for help, support and many fun hours of studying together.

We wish to thank Shay Ben Moshe and Yuval Lotenberg for reading drafts of this work. Thanks to Nathaniel Stapleton for reviewing this work and providing helpful remarks. 

\section{Ambidextirety and the definition of $\delta$}\label{sec:ambyanddelta}

Let $\cat{C}$ be a category with small limits and colimits. We say that $\cat{C}$ is $-2$ semiadditive without any requirements. 

Suppose we have a $-2$ semiadditive category, then any object $X\in \cat{C}$ is represented by a functor $\explicitset{*} \rightarrow \cat{C}$.
We have a canonical map $\nm_{\explicitset{*}}:X = \colim_{\explicitset{*}} X \rightarrow \lim_{\explicitset{*}} X = X$ which is homotopic to the identity. We also have a canonical map, homotopic to the identity as well, given by the composition \[\int_{\explicitset{*}}\id_{X} : X \xrightarrow{\Delta} \lim_{\explicitset{*}} X \xrightarrow{\nm_{\explicitset{*}}^{-1}} \colim_{\explicitset{*}} X \xrightarrow{\nabla} X.\]

If $\cat{C}$ is $-2$ semiadditive then there is a canonical map $\nm_{\emptyset}: \emptyset \rightarrow *$ from the initial object to the terminal object obtained from their universal properties.
We say that $\cat{C}$ is $-1$  semiadditive if $\nm_{\emptyset}$ is an isomorphism. Such a category is also occasionally called "pointed" in the literature.

Suppose we have a $-1$ semiadditive category $\cat{C}$. This allows to define a special morphism between any two objects $X,Y\in C$ called the zero map, defined as \[\int_{\emptyset} : X\rightarrow * \xrightarrow{\sim} \emptyset \rightarrow Y.\]

Let $\mathcal{C}$ be $-1$ and let $A$ be a finite set. One defines for any $A$ indexed collection of objects  $X_a\in \cat{C}$ for $a\in A$, a natural map $\coprod_{a\in A} X_a \rightarrow \prod_{a\in A} X_a $. Giving such a map is the same as giving a matrix indexed by $A^2$ of morphisms $X_a\to X_b$ for $(a,b)\in A^2$. For this matrix we take the identity matrix \[\parenth{\delta^i_j}_{(i,j)\in A^2}, \delta^i_j = \begin{cases}\int_{\explicitset{*}}\id_{X} & i=j \\ \int_\emptyset & i\neq j \end{cases}\]
and obtain a map which we will call $\nm_A$. 

We say that $\cat{C}$ is $0$-semiadditive (or just semiadditive) if for any finite set $A$, $\nm_A$ is a natural isomorphism. This enables one to define an integration map $\int_A$ which takes a map $\phi:A\rightarrow \hom(X,Y)$ and returns $\int_A\phi:X\rightarrow Y$ defined by the composition \[X\xrightarrow{\Delta} \prod_{a\in A} X \xrightarrow{\prod_{a\in A} \phi(a)} \prod_{a\in A} Y \xrightarrow{\nm^{-1}} \coprod_{a\in A} Y \xrightarrow{\nabla} Y. \]

This integration over finite sets endows the mapping spaces in $\cat{C}$ with a structure of a commutative monoid where the unit is $\int_{\emptyset}$. This structure allows to define, given an action of a finite group $G$ on an object $X \in \cat{C}$, a morphism $\nm_{BG}:X_{G}\rightarrow X^{G}$. A category where all norm maps $\nm_{BG}$ are equivalances is called $1$ semiadditive.

This story is part of an inductive sequence of definitions one can make. Given an $(n-1)$-semiadditive category, one has integration over $(n-1)$-finite spaces, which allows one to define an $n$-norm map on $n$-finite spaces and then define a notion of an $n$-semiadditive category, proceeding ad infinitum. We will explicate on this in subsection \S \ref{suse:inftysemiadd}.

Given a $1$ semiadditive symmetric monoidal category $\cat{C}$, a commutative coalgebra $C$ and a commutative algebra $A$ one can define a natural transformation $\alpha: \pi_0 \map(C,A)\rightarrow \pi_0 \map(C,A)$ which is dependent on both the algebraic structures on $A,C$ and on the notion of higher semiadditivity of $\cat{C}$. Stating this definition and going over some of $\alpha$'s basic properties is the goal of subsection \S\ref{suse:thealphaop}.

Another thing one can do with an $\infty$-semiadditive category is to define a notion of cardinality of $\pi$-finite spaces. Given a $\pi$-finite space $A$ and an object $X\in \cat{C}$, one can define $ X^A= \lim_{A}X, X_A = \colim_{A}X$. Then we get a natural transformation of the identity functor in $\cat{C}$ by the composition $\cardinality{A}_X = \int_A \id_X: X\rightarrow X^A\xrightarrow{\nm_A^{-1}} X_A\rightarrow X$. We call this natural transformation the \emph{cardinality} of $A$. In \S \ref{suse:cardanddelta} we will give a definition of this notion and combine it with the $\alpha$ operation to get a power operation. This is the $\delta$ operation.

\subsection{$\infty$-semiadditivity}\label{suse:inftysemiadd}

\begin{definition}
A space $A$ is called $m$-finite if it has finitely many connected components, each with finite homotopy groups, and if $\pi_n(A)=0$ for all $n>m$. A space is called $\pi$-finite if it is $m$-finite for some $m \geq -2$.
\end{definition}

Thus a $(-2)$-finite space is contractible, under the convention that $\pi_{-1}$ is empty if the space is empty and a singleton if the space is inhabited. A $(-1)$-finite space is either contractible or empty. A $0$-finite space is a finite disjoint union of contractible spaces. A $1$-finite space is a finite disjoint union of Eilenberg--MacLane classifying spaces $BG$ of finite groups.

\begin{definition}
A map of spaces $f \colon A \to B$ is called $m$-finite if all of its fibres are $m$-finite.
\end{definition}

Thus a $(-2)$-finite map is a homotopy equivalence, a $(-1)$-finite map is an inclusion of connected components, and a $0$-finite map is a finite covering map.

\begin{proposition}\label{prop:2-3}
Let $X \to Y \to Z$ be a fibre sequence of spaces with $Y$ and $Z$ connected. If both $Y$ and $Z$ are $m$-finite, then so is $X$.  If both $X$ and $Z$ are $m$-finite, then so is $Y$. 
\end{proposition}

\begin{proof}
This follows from the exactness of the Puppe sequence.
\end{proof}
\begin{corollary}\label{cor:closedunderpb}
$m$-finite spaces are closed under pullbacks.
\end{corollary}

\begin{proof}
Let
\[
\begin{tikzcd}
	X & Y \\
	Z & W
	\arrow[from=1-2, to=2-2]
	\arrow[from=1-1, to=2-1]
	\arrow[from=1-1, to=1-2]
	\arrow[from=2-1, to=2-2]
	\arrow["\lrcorner"{anchor=center, pos=0.125}, draw=none, from=1-1, to=2-2]
\end{tikzcd}
\]
be a pullback square of spaces, and assume that $Y$, $Z$, and $W$ are $m$-finite. By Proposition~\ref{prop:2-3}, each fibre of the map $Z \to W$ is $m$-finite. Each fibre of $X \to Y$ is canonically homotopy equivalent to a finite disjoint union of fibres of $Z \to W$, and is therefore $m$-finite. Applying Proposition~\ref{prop:2-3} to each connected component of $X$ shows that $X$ is $m$-finite.
\end{proof}
\begin{lemma}\label{lemma:diagismorefinite}
Let $f \colon A \to B$ be an $m$-finite map. Then its diagonal
\[
\Delta_f \colon A \to A \times_B A,
\]
defined by the diagram
\[
\begin{tikzcd}
	A \\
	& {A\times_B A} & A \\
	& A & B
	\arrow["f", from=2-3, to=3-3]
	\arrow["f"', from=3-2, to=3-3]
	\arrow["{\pi_2}"', from=2-2, to=3-2]
	\arrow["{\pi_1}", from=2-2, to=2-3]
	\arrow[curve={height=-12pt}, from=1-1, to=2-3, equal]
	\arrow[curve={height=12pt}, from=1-1, to=3-2, equal]
	\arrow["{\Delta_f}"{description}, from=1-1, to=2-2]
\end{tikzcd}
\]
is $(m-1)$-finite.
\end{lemma}

\begin{proof}
By \cite[Proposition~5.5.6.15]{HTT}, the map $\Delta_f$ is $(m-1)$-truncated. By Corollary~\ref{cor:closedunderpb}, the space $A \times_B A$ is $m$-finite, and hence the map $\Delta_f$ is $m$-finite. Combining these two facts, we conclude that $\Delta_f$ is $(m-1)$-finite.
\end{proof}

 \begin{definition}
Let $f \colon A \to B$ be a map of spaces, and let $\cat{C}$ be a category that admits limits and colimits indexed by the fibres of $f$. Then there is a pullback functor
\[
f^* \colon \fun(B,\cat{C}) \to \fun(A,\cat{C}).
\]
Left and right Kan extensions along $f$ define left and right adjoints to $f^*$, which we denote by $f_!$ and $f_*$, respectively.
\end{definition}

\begin{definition}
Let $\cat{C}$ be a category. We say that $\cat{C}$ is $(-2)$-semiadditive without imposing any conditions.
\end{definition}

\begin{remark}
In particular, a $(-2)$-semiadditive category admits limits and colimits indexed by $(-2)$-finite spaces. Moreover, for a $(-2)$-finite map $f \colon A \to B$, the functor
\[
f^* \colon \fun(B,\cat{C}) \to \fun(A,\cat{C})
\]
is an equivalence. In this case, its left and right adjoints are both inverses of $f^*$ and are therefore naturally isomorphic. We denote this natural isomorphism by
\[
\nm_{\explicitset{*}} \colon f_! \to f_*.
\]
\end{remark}

Although this case is trivial, these definitions allow us to begin the inductive construction of higher semiadditivity. We first define a $k$-finite norm map in a $(k-1)$-semiadditive category.

\begin{definition}\label{def:integration}
Let $\cat{C}$ be a $(k-1)$-semiadditive category. Then $\cat{C}$ admits limits and colimits indexed by $(k-1)$-finite spaces, and for any $(k-1)$-finite map $f \colon A \to B$ there is a norm isomorphism
\[
\nm_f \colon f_! \to f_*.
\]
Given objects $X,Y \in \fun(B,\cat{C})$ and a morphism $a \colon f^*X \to f^*Y$, we define a morphism
\[
\int_f a \colon X \to Y
\]
as the composite
\[
\begin{tikzcd}
	X & {f_*f^*X} & {f_*f^*Y} \\
	& {f_!f^*X} & {f_!f^*Y} & Y
	\arrow[from=1-1, to=1-2]
	\arrow["{f_*a}", from=1-2, to=1-3]
	\arrow["{f_!a}"', from=2-2, to=2-3]
	\arrow[from=2-3, to=2-4]
	\arrow["{\nm_f^{-1}}"', from=1-2, to=2-2]
	\arrow["{\nm_f^{-1}}", from=1-3, to=2-3]
\end{tikzcd}
\].
\end{definition}

\begin{example}
If $\cat{C}$ is $(-2)$-semiadditive, then $\int_f$ is the inverse of
\[
f^* \colon \map(X,Y) \to \map(f^*X,f^*Y).
\]
\end{example}

\begin{definition}\label{def:normmap}
Let $\cat{C}$ be a $(k-1)$-semiadditive category that admits $k$-finite limits and colimits. Let $f \colon X \to Y$ be a $k$-finite map. By Lemma~\ref{lemma:diagismorefinite}, the diagonal $\Delta_f$ is $(k-1)$-finite, and hence $(k-1)$-semiadditivity provides a natural isomorphism
\[
\nm_{\Delta_f} \colon (\Delta_f)_! \to (\Delta_f)_*.
\]
Since $\Delta_f^*\pi_2^* = \id^* = \Delta_f^*\pi_1^*$, we obtain a morphism
\[
\id \in \map(\Delta_f^*\pi_2^*, \Delta_f^*\pi_1^*).
\]
Applying Definition~\ref{def:integration}, this defines a morphism
\[
\int_{\Delta_f} \id \colon \pi_2^* \to \pi_1^*.
\]
We define the norm map associated to $f$ as the composite
\[
\nm_f \colon
f_! \to f_*f^*f_!
\xrightarrow{\sim} f_*(\pi_1)_!\pi_2^*
\xrightarrow{\int_{\Delta_f} \id} f_*(\pi_1)_!\pi_1^*
\to f_*,
\]
where the first and last maps are the unit and counit of the adjunctions, and the middle equivalence is given by Beck--Chevalley.
\end{definition}

\begin{example}
Let $f \colon \emptyset \to *$ be the unique morphism in $\spaces$. The map $f$ is $(-1)$-finite. The functor
\[
f_! \colon \cat{C}^{\emptyset} = * \to \cat{C} = \cat{C}^{*}
\]
selects an initial object of $\cat{C}$, while
\[
f_* \colon \cat{C}^{\emptyset} = * \to \cat{C} = \cat{C}^{*}
\]
selects a terminal object. The norm map $\nm_f$ is therefore the unique morphism from the terminal object to the initial object in $\cat{C}$. A $(-1)$-semiadditive category is precisely a category in which this morphism is an isomorphism. Such a category is also known as a pointed category.
\end{example}

\begin{example}
Let $A$ be a finite set and let $f \colon A \to *$. Then the diagonal
\[
\Delta_f \colon A \to A \times A
\]
is an inclusion of connected components. Integrating the identity along $\Delta_f$, the map $\int_{\Delta_f} \id$ is equal to $\id$ on each point in the image and to $0$ on points not in the image. Unwinding Definition~\ref{def:integration}, it follows that the norm map $\nm_f$ is the identity matrix. For example, in $\spaces_*$, which is $(-1)$-semiadditive, the norm map associated to $*\amalg * \to *$ is the canonical map
\[
X \vee Y \to X \times Y.
\]
\end{example}

\begin{definition}
Let $\cat{C}$ be a $(k-1)$-semiadditive category that admits $k$-finite limits and colimits. We say that $\cat{C}$ is $k$-semiadditive if the norm map $\nm_f$ is an isomorphism for every $k$-finite map $f$. We say that $\cat{C}$ is $\infty$-semiadditive if it is $k$-semiadditive for all $k$.
\end{definition}

\begin{example}
The category $\tmnt{n}$ is $\infty$-semiadditive.
\end{example}

\begin{proof}
By \cite[Theorem~5.2.1]{HL}, the category $\spectra_{K(n)}$ is $\infty$-semiadditive. The functor
\[
E_n \otimes - \colon \spectra_{K(n)} \to \tmnt{n}
\]
is a symmetric monoidal left adjoint, and therefore satisfies the hypotheses of \cite[Corollary~3.3.2]{CSY1}. It follows that $\tmnt{n}$ is $\infty$-semiadditive.
\end{proof}

\begin{notation}
If $A$ is a $\pi$-finite space, let $\pi_A \colon A \to *$ denote the unique map. We write $\nm_A = \nm_{\pi_A}$ and $\int_A = \int_{\pi_A}$.
\end{notation}

\subsection{The $\alpha$ operation}\label{suse:thealphaop}

$\infty$-semiadditivity allows one to sum over $\pi$-finite families of morphisms. In algebra, we are often interested in algebras, and in particular in multiplication operations. A basic requirement is that multiplication distributes over addition.

In higher algebra, the standard setting for this compatibility is a commutative algebra object in a symmetric monoidal stable category whose tensor product preserves colimits in each variable.

In the previous chapter, we defined a notion of integration over $\pi$-finite spaces. It is natural to seek a notion of algebra in which multiplication distributes over these integration maps, in the sense that
\[
b \cdot \int_A a = \int_A b \cdot a.
\]
This compatibility is formalized in the following definition.

\begin{definition}
An $m$-semiadditively symmetric monoidal category $(\cat{C},\otimes)$ is a symmetric monoidal $m$-semiadditive category $(\cat{C},\otimes)$ such that the tensor product $\otimes$ preserves $m$-finite colimits in each variable.
\end{definition}

\begin{remark}
By $m$-semiadditivity, this condition is equivalent to requiring that the tensor product $\otimes$ preserves $m$-finite limits in each variable.
\end{remark}

\begin{example}
The category $\tmnt{n}$ is presentable and admits a symmetric monoidal structure such that the functor $X \otimes -$ has a right adjoint for every object $X$. As noted above, $\tmnt{n}$ is $\infty$-semiadditive, and therefore $\infty$-semiadditively symmetric monoidal.
\end{example}

Let $\cat{C}$ be a $1$-semiadditively symmetric monoidal category. One consequence of $m$-semiadditivity is the existence of power operations on commutative algebras. In this section, we define one such operation. Let $L,R \in \cat{C}$ be a commutative coalgebra and a commutative algebra, respectively. Then
\[
\pi_0 \map(L,R)
\]
has the structure of a commutative semiring. If $\cat{C}$ is additive, this is a commutative ring.

\begin{definition}
Let $L,R \in \cat{C}$ be a cocommutative coalgebra and a commutative algebra, respectively, in a $1$-semiadditively symmetric monoidal category. Given a morphism $x \colon L \to R$, consider the map
\[
x^{\otimes p} \colon L^{\otimes p} \to R^{\otimes p},
\]
viewed as a morphism in $\fun(BC_p,\cat{C})$ via the cyclic action. The $\alpha$ power operation
\[
\alpha \colon \pi_0 \map(L,R) \to \pi_0 \map(L,R)
\]
sends $x$ to the composite
\[
\begin{tikzcd}
	L && {(L^{\otimes p})^{hC_p}} && {(R^{\otimes p})^{hC_p}} \\
	\\
	&& {(L^{\otimes p})_{hC_p}} && {(R^{\otimes p})_{hC_p}} && R.
	\arrow["{(x^{\otimes p})^{hC_p}}", from=1-3, to=1-5]
	\arrow["{\nm_{BC_p}^{-1}}"', from=1-3, to=3-3]
	\arrow["{(x^{\otimes p})_{hC_p}}"', from=3-3, to=3-5]
	\arrow["{\nm_{BC_p}^{-1}}", from=1-5, to=3-5]
	\arrow["\text{product}", from=3-5, to=3-7]
	\arrow["\text{coproduct}", from=1-1, to=1-3]
\end{tikzcd}
\].
\end{definition}

\begin{example}
We are interested in the case where $\cat{C} = \tmnt{n}$ and $L = E$. In this situation,
\[
\pi_0 \map_{\tmnt{n}}(E,R) = \pi_0 \map_{\spectra}(\mathbb{S},R) = \pi_0(R),
\]
so the operation $\alpha$ defines a map
\[
\alpha \colon \pi_0(R) \to \pi_0(R).
\]
\end{example}

We recall some basic properties of the $\alpha$ operation.

\begin{proposition}\cite[Proposition~4.2.8]{CSY1}
Let $L,R \in \cat{C}$ be a commutative coalgebra and a commutative algebra, respectively, in a $1$-semiadditively symmetric monoidal category. Let $x,y \colon L \to R$. Then:
\begin{enumerate}
    \item $\alpha(x+y) - \alpha(x) - \alpha(y) = \frac{(x+y)^p - x^p - y^p}{p}$.
    \item $\alpha(0) = 0$.
\end{enumerate}
\end{proposition}

We focus on a particular case.

\begin{proposition}\cite[Lemma~4.2.10]{CSY1}
Suppose both $L,R$ are the unit objects of $\cat{C}$, equipped with their canonical coalgebra and algebra structures. Then
\[
(L^{\otimes p})^{hC_p} \simeq L^{BC_p},
\]
and the map $L \to L^{BC_p}$ induced by the comultiplication identifies with the unit
\[
\id \to \pi_*\pi^*.
\]
Similarly,
\[
(R^{\otimes p})_{hC_p} \simeq R \otimes BC_p,
\]
and the map $(R^{\otimes p})_{hC_p} \to R$ induced by the multiplication identifies with the counit
\[
\pi^*\pi_! \to \id.
\]
\end{proposition}

\begin{corollary}
If $L$ and $R$ are unit objects of $\cat{C}$, then
\[
\alpha(f) = \int_{BC_p} f^{\otimes p}.
\]
\end{corollary}

\subsection{Cardinality and the $\delta$ operation}\label{suse:cardanddelta}

In a semiadditive $1$-category, every object carries a canonical structure of a commutative monoid. The addition map is given by $X \oplus X \to X$. There is a $\pi$-finite analogue of this construction in our setting. We first introduce precise terminology for commutative monoids.

\begin{definition}
An $m$-commutative monoid in a category $\cat{C}$ admitting $m$-finite limits is a functor
\[
F \colon \spaan(\spaces_{m\text{-}\fin}) \to \cat{C}
\]
that preserves $m$-finite limits. We denote the category of $m$-commutative monoids in $\cat{C}$ by $\cmon_m(\cat{C})$.

An $\infty$-commutative monoid is an object of the limit
\[
\lim_m \cmon_m(\cat{C}).
\]
\end{definition}

\begin{remark}
The natural functor
\[
\fun^{m\text{-fin}}(\spaan(\spaces_{\pi\text{-}\fin}), \cat{C}) \to \lim_m \cmon_m(\cat{C})
\]
with the LHS the category of functors that preserves $m$-finite limits, is an equivalence. 
\end{remark}

\begin{example}
A $(-1)$-commutative monoid is a pointed object.
\end{example}

\begin{example}
A $0$-commutative monoid is an ordinary commutative monoid (that is, an $\eoperad_{\infty}$-algebra).
\end{example}

\begin{proposition}\cite[Theorem~4.1]{HAR}\label{prop:canmonoid}
Let $\cat{C}$ be an $m$-semiadditive category. Then the functor
\[
\cmon_m(\cat{C}) \to \cat{C}
\]
given by evaluation at $*$ is an equivalence. Equivalently, every object $X \in \cat{C}$ admits a unique structure of an $m$-commutative monoid.
\end{proposition}
Roughly speaking, the structure is defined by sending an $m$-finite space $A$ to $\lim_A M$, the limit of the constant diagram at $M$, sending wrong-way maps to maps of limits, and sending right-way maps to maps of colimits.

\begin{definition}
The cardinality of $A \in \spaces_{m\text{-}\fin}$ in $M$ is the endomorphism
\[
\cardinality{A}_M = M(* \leftarrow A \rightarrow *) \colon M \to M.
\]
\end{definition}

\begin{proposition}
Let $M \in \cat{C}$ be an $m$-commutative monoid in an $m$-semiadditive category. Then
\[
\cardinality{A}_M = \int_A \id.
\]
Consequently, $\cardinality{A}_-$ defines a natural endomorphism of the identity functor on $\cat{C}$.
\end{proposition}

\begin{proof}
For the first claim, the integral $\int_A \id$ is the composite
\[
M \longrightarrow \lim_A M \longrightarrow \colim_A M \longrightarrow M,
\]
which is precisely the definition of $M(* \leftarrow A \rightarrow *)$.

The second claim follows because the diagonal and codiagonal maps
\[
M \to \lim_A M
\quad\text{and}\quad
\colim_A M \to M
\]
are natural in $M$, as is the norm map.
\end{proof}

In fact, if $\cat{C}$ is $m$-semiadditively symmetric monoidal, we can say something stronger.

\begin{proposition}\cite[Lemma~3.3.4]{CSY1}
Let $\cat{C}$ be an $m$-semiadditively symmetric monoidal category with unit $\smunit$, and let $M \in \cat{C}$. Then
\[
\cardinality{A}_M = \cardinality{A}_{\smunit} \otimes M.
\]
\end{proposition}

We are now prepared to define the main object of study of this project.

\begin{definition}
Let $L,R \in \cat{C}$ be a commutative coalgebra and a commutative algebra, respectively, in a $1$-semiadditively symmetric monoidal category. Suppose that
\[
S = \pi_0 \map(L,R)
\]
is a ring, for example when $\cat{C}$ is stable. We define an endomorphism
\[
\delta \colon S \to S
\]
by
\[
\delta(f) = \cardinality{BC_p}_R \cdot f - \alpha(f).
\]
\end{definition}

We record a few basic properties of $\delta$.

\begin{proposition}\cite[Theorem~4.3.2]{CSY1}
\begin{enumerate}
    \item $\delta(0) = \delta(1) = 0$.
    \item $\delta(f+g) - \delta(f) - \delta(g) = \frac{f^p + g^p - (f+g)^p}{p}$.
\end{enumerate}
\end{proposition}

\begin{remark}
These properties identify $\delta$ as an additive $p$-derivation. In \cite{CSY1}, the authors use this structure to prove that the $T(n)$-local category is $\infty$-semiadditive.
\end{remark}

\section{Power operations at height $2$}\label{sec:powerop}

The $\alpha$ and $\delta$ operations are examples of power operations. A power operation is an operation on homotopy classes of commutative algebras that depends on the algebra structure.

\begin{definition}
Let $\cat{C}$ be a category of modules over a ring spectrum in a localization of $\spectra$. A power operation of degree $0$ on commutative algebras in $\cat{C}$ is a natural endomorphism of the functor
\[
\pi_0 \in \fun(\calg(\cat{C}), \catname{Set}).
\]
\end{definition}

Given $R \in \calg(\spectra_{K(n)})$, we have
\[
\pi_0(R) = \pi_0 \map_{\cat{C}}(\smunit, R)
= \pi_0 \map_{\calg(\cat{C})}\parenth{\coprod_{m \geq 0} \smunit^{h\Sigma_m}, R}.
\]
By the Yoneda lemma, power operations of degree $0$ are therefore represented by elements of
\[
\bigoplus_{m \geq 0} \pi_0\parenth{\smunit^{h\Sigma_m}}.
\]

The goal of this chapter is to present part of Zhu's calculation in \cite{ZHU} of degree $0$ power operations in $\tmnt{2}$.

The operations $\alpha$ and $\delta$ are degree $0$ power operations defined in any $1$-semiadditively symmetric monoidal category. We will use Zhu's calculation to compute $\alpha$ and $\delta$ in $\tmnt{2}$.

In Section~\ref{suse:traceandint}, we use this geometric perspective to state a formula for integration over $BC_p$.

\subsection{The $C_p$ absolute power operation}


\begin{definition}
Let $R$ be a commutative algebra in $\tmnt{n}$. The $m$th total power operation is a map
\[
\widetilde{\psi^m} \colon \pi_0(R) \to \pi_0\parenth{R^{B\Sigma_m}}
\]
defined as follows. Given a morphism $x \colon E \to R$, we set $\widetilde{\psi^m}(x)$ to be the composite
\[
E \to E^{B\Sigma_m}
\cong (E^{\otimes m})^{h\Sigma_m}
\xrightarrow{x^{\otimes m}} (R^{\otimes m})^{h\Sigma_m}
\to R^{B\Sigma_m}.
\]
\end{definition}

These maps are called total power operations for a reason: all power operations can be constructed from them, although we do not make this precise here. We will instead focus on a variation of this construction.

\begin{definition}
Let $R$ be a commutative algebra in $\tmnt{n}$, and let $G \subset \Sigma_m$ be a subgroup of the symmetric group on $m$ letters. The $G$-total power operation is a map
\[
\widetilde{\psi^G} \colon \pi_0(R) \to \pi_0\parenth{R^{BG}}
\]
defined as follows. Given a morphism $x \colon E \to R$, we define $\widetilde{\psi^G}(x)$ to be the composite
\[
E \to E^{BG}
\cong (E^{\otimes m})^{hG}
\xrightarrow{x^{\otimes m}} (R^{\otimes m})^{hG}
\to R^{BG}.
\]
\end{definition}

\begin{proposition}
Let
\[
\widetilde{\eta} \colon \pi_0(R^{B\Sigma_p}) \to \pi_0(R^{BC_p})
\]
be the map induced by the inclusion $C_p \hookrightarrow \Sigma_p$. Then
\[
\widetilde{\psi^{C_p}} = \widetilde{\eta} \circ \widetilde{\psi^{p}}.
\]
\end{proposition}

\begin{proof}
These maps fit into the following commutative diagram, in which all vertical maps are induced by the inclusion $C_p \hookrightarrow \Sigma_p$. The two vertical composites are $\widetilde{\psi^{p}}(x)$ and $\widetilde{\psi^{C_p}}(x)$, respectively, and the claim follows.

\[
\begin{tikzcd}
	E & {E^{B\Sigma_p}\cong (E^{\otimes p})^{h\Sigma_p}} & {(R^{\otimes p})^{h\Sigma_p}} & {R^{B\Sigma_p}} \\
	E & {E^{BC_p}\cong (E^{\otimes p})^{hC_p}} & {(R^{\otimes p})^{hC_p}} & {R^{BC_p}}
	\arrow[from=1-1, to=2-1]
	\arrow[from=1-2, to=2-2]
	\arrow[from=1-3, to=2-3]
	\arrow[from=1-4, to=2-4]
	\arrow[from=1-1, to=1-2]
	\arrow[from=1-2, to=1-3]
	\arrow[from=1-3, to=1-4]
	\arrow[from=2-1, to=2-2]
	\arrow[from=2-2, to=2-3]
	\arrow[from=2-3, to=2-4]
\end{tikzcd}
\]
\end{proof}

The total power operation is not a ring homomorphism, although it respects multiplication. The obstruction is additivity. One therefore considers the additive total power operation, which is a ring homomorphism. We state its definition here.

\begin{definition}\label{def:transfer}
Let $\cat{C}$ be a $0$-semiadditive category, let $X \in \cat{C}$ be an object, and let $f \colon A \to B$ be a covering map in $\spaces$. The map $f$ is $0$-finite, and hence the counit
\[
c \colon f_! f^* \to \id_{\cat{C}^B}
\]
induces a morphism defined by the composite
\begin{align*}
X^A
&= (\pi_A)_* \pi_A^* X
\simeq (\pi_B)_* f_* f^* \pi_B^* X
\simeq (\pi_B)_* f_! f^* \pi_B^* X
\xrightarrow{c}
(\pi_B)_* \pi_B^* X
\simeq X^B.
\end{align*}
This morphism is called the transfer map.
\end{definition}
\begin{proposition}
In the setting of Definition~\ref{def:transfer}, suppose further that $\cat{C}$ is $\infty$-semiadditive and suppose that $B$ is $\pi$-finite. Then, by Proposition~\ref{prop:2-3}, the space $A$ is also $\pi$-finite, and the transfer map agrees with the map induced by the span
\[
A \xleftarrow{\id} A \xrightarrow{f} B
\]
under the equivalence of Proposition~\ref{prop:canmonoid}.
\end{proposition}

\begin{proof}
This is dual to \cite[Lemma~3.8]{HAR}.
\end{proof}

\begin{definition}
The transfer ideal $T_{\Sigma_m} \subset \pi_0 E^{B\Sigma_m}$ is the ideal generated by the images of the transfer maps induced by all subgroups
\[
\Sigma_r \times \Sigma_t \subset \Sigma_m
\]
with $r+t=m$ and $r,t<m$.
\end{definition}

\begin{proposition}
The composite
\[
\psi^m \colon \pi_0E \to \pi_0 E^{B\Sigma_m} \to \pi_0 E^{B\Sigma_m}/T_{\Sigma_m}
\]
is a ring homomorphism.
\end{proposition}

\begin{proof}
By \cite[p.~22, Proposition~2.2]{HINFTY}, the failure of $\psi^m$ to be additive lies in the transfer ideal.
\end{proof}

When $m=p$ is prime, there is a compatible construction for the subgroup $C_p$.

\begin{definition}
The transfer ideal $T_{C_p} \subset \pi_0 E^{BC_p}$ is the ideal generated by the image of the transfer map induced by the inclusion $\explicitset{e} \subset C_p$.
\end{definition}

\begin{proposition}
Under the map $\widetilde{\eta}$, the image of
\[
T_{\Sigma_p} \subset \pi_0 E^{B\Sigma_p}
\]
lands in
\[
T_{C_p} \subset \pi_0 E^{BC_p}.
\]
Consequently, $\widetilde{\eta}$ induces a ring homomorphism
\[
\eta \colon \pi_0 E^{B\Sigma_p}/T_{\Sigma_p} \to \pi_0 E^{BC_p}/T_{C_p}.
\]
\end{proposition}

\begin{proof}
We show that the composite
\[
E^{B(\Sigma_r \times \Sigma_t)} \longrightarrow E^{B\Sigma_p} \xrightarrow{\;\widetilde{\eta}\;} E^{BC_p},
\]
given by the transfer map followed by $\widetilde{\eta}$, factors through the image of the transfer map
\[
E^{B\explicitset{e}} \longrightarrow E^{BC_p}.
\]

The map $\widetilde{\eta}$ is induced by the span
\[
B\Sigma_p \longleftarrow BC_p \longrightarrow BC_p,
\]
while the transfer map for $E^{B\Sigma_p}$ is induced by the span
\[
B(\Sigma_r \times \Sigma_t) \longleftarrow B(\Sigma_r \times \Sigma_t) \longrightarrow B\Sigma_p.
\]
Computing their composite in the span category yields the diagram

\[
\begin{tikzcd}
	&& {(\Sigma_r\times \Sigma_t) \backslash \Sigma_p/C_p} \\
	& {B\Sigma_r\times \Sigma_t} && {BC_p} \\
	{B\Sigma_r\times \Sigma_t} && {B\Sigma_p} && {BC_p}
	\arrow[from=2-2, to=3-1, equal]
	\arrow[from=2-2, to=3-3]
	\arrow[from=1-3, to=2-2]
	\arrow[from=1-3, to=2-4]
	\arrow[from=2-4, to=3-3]
	\arrow[from=2-4, to=3-5, equal]
	\arrow["\lrcorner"{anchor=center, pos=0.125, rotate=-45}, draw=none, from=1-3, to=3-3]
\end{tikzcd}
\]

The space $(\Sigma_r \times \Sigma_t)\backslash \Sigma_p / C_p$ appearing at the top is the double coset space, with quotients taken in the $\infty$-categorical sense.  
If $H,K \leq G$ are subgroups of a finite group, then the size of a double coset $HgK$ is
\[
\cardinality{H}\,[K : K \cap g^{-1}Hg].
\]
Taking $H = C_p$ and $K = \Sigma_r \times \Sigma_t$, we find that every double coset has size $p r! t!$. Indeed, $\Sigma_r \times \Sigma_t$ contains no transitive elements, while every nontrivial element of $g^{-1} C_p g$ is transitive. Since the action of $C_p$ on $(\Sigma_r \times \Sigma_t)\backslash \Sigma_p$ is free, the double coset space is homotopy equivalent to a finite set of cardinality $\binom{p}{r}/p$.

It follows that the composite span is equal to $\binom{p}{r}/p$ copies of the span
\[
B(\Sigma_r \times \Sigma_t) \longleftarrow * \longrightarrow BC_p.
\]
The right-hand map in this span is the transfer map defining $T_{C_p}$. Hence the image of $T_{\Sigma_p}$ under $\widetilde{\eta}$ is contained in $T_{C_p}$.
\end{proof}

\begin{corollary}
The composite
\[
\psi^{C_p} \colon \pi_0E \to \pi_0 E^{BC_p} \to \pi_0 E^{BC_p}/T_{C_p}
\]
is a ring homomorphism.
\end{corollary}

\begin{proof}
This map is the composite of the ring homomorphisms $\eta$ and $\psi^{p}$.
\end{proof}

The ring homomorphisms $\psi^{C_p}$, $\psi^{m}$, and $\eta$ admit geometric descriptions due to Rezk \cite{REZ}. These descriptions are the subject of the next subsection and will be used to compute the corresponding power operations.

\subsection{The moduli space of elliptic curves and the total power operation in $\tmnt{2}$}

By construction, the ring $\pi_0E$ together with its associated formal group is the universal deformation ring of a height $2$ formal group law. In \cite{STR} and later in \cite{REZ}, Strickland and Rezk gave a detailed analysis of power operations in the categories $\tmnt{n}$. Their results are very general; here we recall only the statements needed for our purposes.

\begin{proposition}\label{prop:rezkstrickland} Let $E$ be Morava $E$-theory of any height $n$.
\begin{enumerate}
    \item\label{prop:rezkstrickland:goodgroups} \cite[Sections~3.6 \& 8.6]{STR}  
    The rings
    \[
    \pi_0 E^{B\Sigma_p},\quad
    \pi_0 E^{B\Sigma_p}/T_{\Sigma_p},\quad
    \pi_0 E^{BC_p},\quad
    \pi_0 E^{BC_p}/T_{C_p}
    \]
    are finite free $\pi_0(E)$-modules.

    \item \cite[Section~9.2]{STR}
    The formal scheme $\spf\parenth{\pi_0 E^{B\Sigma_p}/T_{\Sigma_p}}$ classifies degree $p$ isogenies of the Quillen formal group associated to $E$. Equivalently, a local morphism of rings
    \[
    \pi_0 E^{B\Sigma_p}/T_{\Sigma_p} \to R
    \]
    under $\pi_0(E)$ corresponds to a formal group $G$ over $R$ together with a degree $p$ isogeny
    \[
    Q \times_{\spf \pi_0 E} \spf R \to G,
    \]
    where $Q$ denotes the Quillen formal group over $\pi_0(E)$.

    \item \cite[Proof of Theorem~3.19]{REZ}
    The ring $\pi_0 E^{BC_p}/T_{C_p}$ represents the universal point of order $p$ of the Quillen formal group associated to $E$. Moreover,
    \[
    \pi_0 E^{BC_p} \cong \pi_0(E)[[x]]/[p](x),
    \]
    where $[p](x)$ is the $p$-series of the Quillen formal group.

    \item\label{prop:rezkstrickland:agreeswitheta}
    The ring homomorphism
    \[
    \eta \colon \pi_0 E^{B\Sigma_p}/T_{\Sigma_p} \to \pi_0 E^{BC_p}/T_{C_p},
    \]
    induced by the inclusion $C_p \hookrightarrow \Sigma_p$, sends a point of order $p$ given by $x$ to the isogeny
    \[
    Q \to Q/x,
    \]
    where $Q/x$ denotes the quotient of $Q$ by the order $p$ subgroup generated by $x$.
\end{enumerate}
\end{proposition}

\begin{proof}[Proof of \ref{prop:rezkstrickland:agreeswitheta}]
See \cite[Proposition~8.1]{StapletonHuan} for a much more general statement which is closer to our language than  older treatments.

\end{proof}

This provides a concrete method for computing these maps for the case of height $n=2$. Let $C/\finitefield_q$ be a supersingular elliptic curve, and let $X$ denote the moduli space of elliptic curves. We write
\[
[C] \colon \spec \finitefield_q \to X
\]
for the morphism classifying $C$.

Let $X_0(p)$ be the moduli space of elliptic curves equipped with a finite subgroup of order $p$ (equivalently, a pair of elliptic curves $C,C'$ together with an isogeny $C \to C'$ of degree $p$), and let $X_1(p)$ be the moduli space of elliptic curves equipped with a point of order $p$.

There are natural morphisms
\[
X_0(p) \to X
\quad\text{and}\quad
X_1(p) \to X
\]
forgetting the level structure. There is also a morphism
\[
X_1(p) \to X_0(p)
\]
sending a point of order $p$ to the subgroup it generates.

The completion of $X$ at $C$, denoted $X_C^\wedge$, is given by $\spf \pi_0(E)$, since $\pi_0(E)$ is the universal deformation ring of a supersingular elliptic curve, equivalently of its formal group law. It follows that:

\begin{corollary}\label{cor:complofmodcurves}
\begin{enumerate}
    \item The formal scheme $\spf\parenth{\pi_0 E^{B\Sigma_p}/T_{\Sigma_p}}$ is the completion of $X_1(p)$ at the preimage of $[C]$ under the morphism $X_1(p)\to X$.

    \item The ring $\pi_0 E^{BC_p}/T_{C_p}$ is the completion of $X_0(p)$ at the preimage of $[C]$ under the morphism $X_0(p)\to X$.

    \item The ring homomorphism
    \[
    \eta \colon \pi_0 E^{B\Sigma_p}/T_{\Sigma_p} \to \pi_0 E^{BC_p}/T_{C_p}
    \]
    is the completion of the morphism $X_1(p)\to X_0(p)$ at the preimages of $[C]$.
\end{enumerate}
\end{corollary}

\begin{proof}
Since $C$ is supersingular, the deformation theory of $C$ is controlled entirely by its formal group. The formal completions of $X$, $X_0(p)$, and $X_1(p)$ at the points lying over $[C]$ therefore classify, respectively, deformations of the formal group of $C$, finite subgroup schemes of order $p$, and points of order $p$ in these deformations.

By Proposition~\ref{prop:rezkstrickland}, the rings $\pi_0 E^{B\Sigma_p}/T_{\Sigma_p}$ and $\pi_0 E^{BC_p}/T_{C_p}$ represent precisely these deformation problems. Compatibility of the forgetful maps with completion then identifies $\eta$ with the completed morphism $X_1(p)\to X_0(p)$.
\end{proof}

\subsection{The trace map and integration}\label{suse:traceandint}

To compute the $\delta$ operation on $E$, we first study integration over $BC_p$ with coefficients in $E$. For this, we recall the algebraic notion of trace.

\begin{definition}
    Let $A$ be a commutative ring, and let $R$ be an $A$-algebra whose underlying $A$-module is finite projective. The trace map
    \[
    \tr_{R/A} \colon R \to A
    \]
    is the $A$-linear map defined as the composite
    \[
    R \xrightarrow{a} \en_A(R) \cong R \otimes_A R^\vee \to A,
    \]
    where $a$ is the action of $R$ on itself by left multiplication and $R^\vee = \hom_A(R,A)$.
\end{definition}

We will need the following identification to perform our computation.

\begin{proposition}\label{prop:trasspan}
    The map
    \[
    \tau \colon \pi_0 E^{BC_p} \to \pi_0(E)
    \]
    induced by the span
    \[
    BC_p \xleftarrow{\id} BC_p \to *
    \]
    is equal to $\frac{1}{p}\,\tr_{\pi_0 E^{BC_p}/\pi_0(E)}$.
\end{proposition}

\begin{proof}
We recall that $\pi_0 E^{BC_p}$ is a Hopf algebra classifying the $p$-torsion subgroup of the Quillen formal group, and that it is torsion-free (see Proposition~\ref{prop:rezkstrickland}). By \cite[Theorem~3.38]{MILNE}, there exists a faithfully flat $\pi_0(E)$-algebra $C$ such that
\[
\pi_0 E^{BC_p} \otimes_{\pi_0(E)} C
\]
is isomorphic to the Hopf algebra classifying a constant group
\[
G \cong (\integers/p)^n.
\]

This Hopf algebra is the algebra of functions $G \to C$. For each $g \in G$, let $e_g$ denote the indicator function of $g$. The elements $\{e_g\}_{g\in G}$ form a basis of $\pi_0 E^{BC_p} \otimes_{\pi_0(E)} C$. The unit is therefore $\sum_{g\in G} e_g$, and multiplication is pointwise. The coproduct is given by
\[
e_g \longmapsto \sum_{h_1+h_2=g} e_{h_1}\otimes e_{h_2},
\]
and the counit is
\[
\sum_{g\in G} a_g e_g \longmapsto a_0.
\]

The coalgebra structure on $\pi_0 E^{BC_p}$ arises from the group structure on $BC_p$, while the algebra structure comes from the cohomology ring structure. Thus, the span
\[
* \leftarrow BC_p \rightarrow BC_p
\]
corresponds to the unit of the algebra. Likewise, the comultiplication
\[
m \colon \pi_0 E^{BC_p} \to \pi_0 E^{BC_p} \otimes \pi_0 E^{BC_p}
\]
is induced by the span
\[
BC_p \xleftarrow{+} BC_p \times BC_p \to BC_p \times BC_p,
\]
where the left map is the group law on $BC_p$.

Write
\[
\tau(e_g) = c_g \in C.
\]
Since both $\tau$ and $\frac{1}{p}\tr$ are $\pi_0(E)$-linear, it suffices to show that $c_g = 1/p$ for all $g\in G$.

Consider the following $2$-isomorphism of morphisms in $\spaan(\spaces_{1-\fin})$:
\[
\begin{tikzcd}
	& {BC_p\times BC_p} \\
	{BC_p} && {BC_p} \\
	& {BC_p\times BC_p}
	\arrow["{+}"', from=1-2, to=2-1]
	\arrow["{\pi_2}", from=1-2, to=2-3]
	\arrow["{\pi_1}", from=3-2, to=2-1]
	\arrow["{\pi_2}"', from=3-2, to=2-3]
	\arrow["s"{description}, from=1-2, to=3-2]
\end{tikzcd}
\]
where $s$ is the shear map, defined by addition in the first coordinate and projection to the second. Since $BC_p$ is a group object, $s$ is an isomorphism.

By symmetric monoidal functoriality of $E^{(-)}$, the maps induced by these two spans are homotopic, and products are sent to tensor products (this is the Künneth property in $\tmnt{2}$). The upper span induces the map $(\tau\otimes \id)\circ m$, while the lower span induces $\tau\otimes u$, where $u$ is the unit of the algebra.

For any $h\in G$, we therefore obtain
\begin{align*}
(\tau\otimes \id)\circ m(e_h)
&= (\tau\otimes \id)\parenth{\sum_{g_1+g_2=h} e_{g_1}\otimes e_{g_2}}
 = \sum_{g_1+g_2=h} c_{g_1} e_{g_2}, \\
(\tau\otimes u)(e_h)
&= \sum_{g\in G} c_h\, e_g.
\end{align*}
Comparing coefficients of the basis $\{e_g\}$ shows that all $c_g$ are equal.

It remains to compute $c_0$. Since
\[
(\tau\otimes C)(1) = \sum_{g\in G} c_0 = \cardinality{G}\, c_0,
\]
it suffices to compute $\tau(1)$. By definition, $\tau(1)$ is the cardinality
\[
\cardinality{BC_p} \colon
E \xrightarrow{1} E^{BC_p} \xrightarrow{\sim} E\otimes BC_p \to E,
\]
which equals $p^{n-1}$ by \cite[Lemma~5.3.4]{CSY1}. Since $\cardinality{G}=p^n$, we conclude that $c_0 = 1/p$, and hence $c_g=1/p$ for all $g\in G$.
\end{proof}

\begin{proposition}\label{prop:alphaistraceoftotalpo}
    Let $\tilde{\psi^{C_p}}:\pi_0E\rightarrow \pi_0 E^{BC_p}$ be the total power operation of $C_p$, then $\alpha=\frac{1}{p}\tr\circ\; \tilde{\psi^{C_p}}$.
\end{proposition}
\begin{proof}
    By \ref{prop:trasspan} we have that the map $\pi_0R^{BC_p}\rightarrow R_0$ is the trace. $\alpha$ was defined as the composition of this map with the total power operation.
\end{proof}

\section{The character of the total power operation}\label{sec:charoftoalpower}

We now have most of the ingredients needed to state a formula for $\alpha$ on $\pi_0E$ at height $2$ and prime $p=3$. In Proposition~\ref{prop:alphaistraceoftotalpo} we saw that
\[
\alpha(x) = \int_{BC_p} \widetilde{\psi^{C_p}}(x).
\]
The operation $\psi^{C_3}$ is already understood. The remaining task is to analyze the part of $\widetilde{\psi^{C_3}}$ that is killed upon passage to the quotient by the transfer ideal.

In this section, we prove the following result.

\begin{proposition}\label{thm:charactercomputation}
There is a natural isomorphism
\[
\pi_0E^{B\Sigma_p} \otimes \rationals
\cong
\pi_0E \otimes \rationals
\times
\parenth{\pi_0E^{B\Sigma_p}/T_{\Sigma_p} \otimes \rationals}.
\]
Under this identification, the rationalization of the total power operation acts by
\[
x \longmapsto x^p
\]
on the first factor, and by the rationalization of $\psi^p$ on the second factor.
\end{proposition}

To establish this decomposition, we use HKR character theory. HKR character theory provides a description of the rationalized Morava $E$-theory of finite groups in terms of group-theoretic data. We apply it here to compute the rationalization of the total power operation, which in turn allows us to extract the trace of $\widetilde{\psi^{C_p}}$.

\subsection{HKR character theory}

Let $G$ be a finite group.

\begin{definition}\label{def:nclasses}
    Let $G_n$ denote the set of $n$-tuples of pairwise commuting $p$-torsion elements of $G$. Equivalently,
    \[
    G_n = \hom(\integers_p^n, G).
    \]
    This set carries a natural action of $G$ by conjugation.
\end{definition}

\begin{definition}
    The height $n$ generalized class functions of $G$ with coefficients in a ring $C$ are the $C$-algebra of functions on the quotient
    \[
    \hom(\integers_p^n, G)/G,
    \]
    where $G$ acts by conjugation. We denote this algebra by $\cl_n(G,C)$. If $R$ is equipped with the trivial $G$-action, then
    \[
    \hom_{G\text{-}\catname{Set}}(G_n,R) = \cl_n(G,R).
    \]
\end{definition}

\begin{example}
    Since $C_p$ is an abelian $p$-group, conjugation is trivial. Therefore $\cl_2(C_p)$ is a free module of rank
    \[
    \cardinality{\hom(\integers_p^2,C_p)} = \cardinality{C_p^2} = p^2.
    \]
    Indeed, for any ordered pair of elements of $C_p$ there is a unique group homomorphism
    \[
    \integers_p^2 \to C_p
    \]
    sending $(1,0)$ and $(0,1)$ to the given elements.

    Note that this agrees with the rank of $\pi_0E^{BC_p}$ in height $n=2$.
\end{example}

\begin{example}
    The $C$-module $\cl_2(\Sigma_p)$ has rank $p+2$. Indeed, to specify a homomorphism
    \[
    \integers_p^2 \to \Sigma_p
    \]
    we first choose the image of $(1,0)$. This element may be trivial or nontrivial.

    If $(1,0)$ is trivial, then $(0,1)$ is either trivial (one possibility) or a $p$-cycle (one additional possibility, unique up to conjugation).

    If $(1,0)$ is nontrivial, then by conjugation we may assume it is the cycle $(1\,2\,\dots\,p)$. In this case $(0,1)$ must be sent to a power of this cycle, giving $p$ distinct possibilities.

    Altogether, this yields $1+1+p = p+2$ conjugacy classes, and hence $\cl_2(\Sigma_p)$ has rank $p+2$.

    Note that this also agrees with the rank of $\pi_0 E^{B\Sigma_p}$ in height $n=2$.
\end{example}

HKR character theory identifies a suitable base change of $\pi_0 E^{BG}$ with $\cl_n(G,C)$, providing a powerful computational tool. We now formulate this theory precisely.

Over an algebraically closed field of characteristic $0$, all formal groups are étale. Thus there exists a faithfully flat $\pi_0(E)$-algebra over which the Quillen formal group $Q$ becomes isomorphic to $(\rationals_p/\integers_p)^n$. There is a universal algebra carrying such an isomorphism in the following sense.

\begin{definition}\cite[Corollary~6.8(i)]{HKR}
    There exists a rational $\pi_0(E)$-algebra $C_0$, called the splitting algebra, such that for any rational $\pi_0(E)$-algebra $R$ there is a natural identification
    \[
    \hom_{\pi_0(E)\text{-}\catname{Alg}}(C_0,R)
    =
    \iso\parenth{(\rationals_p/\integers_p)^n,\; Q\times \spf R}.
    \]
\end{definition}

\begin{proposition}\label{prop:actionongln}
    The algebra $C_0$ carries a natural left action of $Gl_n(\integers_p)$.
\end{proposition}

\begin{proof}
    By the Yoneda lemma, giving a left action of $Gl_n(\integers_p)$ on $C_0$ is equivalent to giving a right action on the functor
    \[
    R \longmapsto \iso\parenth{(\rationals_p/\integers_p)^n,\; Q\times \spf R}.
    \]
    By the Yoneda lemma again, this is equivalent to giving a left action of $Gl_n(\integers_p)$ on $(\rationals_p/\integers_p)^n$. Such an action is canonical, since
    \[
    (\rationals_p/\integers_p)^n \cong \hom(\integers_p^n,\rationals_p/\integers_p),
    \]
    and $Gl_n(\integers_p)$ acts on $\integers_p^n$ by change of basis.
\end{proof}

\begin{proposition}\cite[Corollary~6.8(iii)]{HKR}
    The subring of $Gl_n(\integers_p)$-fixed points of $C_0$ is equal to $\pi_0(E)\otimes \rationals$.
\end{proposition}

Thus there is a natural $GL_n(\integers_p)$-action on $\cl_n(G,C_0)$, given by the evident action on $G_n$ together with the action on scalars coming from Proposition~\ref{prop:actionongln}.

This allows us to formulate the main theorem of HKR character theory.

\begin{theorem}\cite[Theorem~C]{HKR}\label{thm:HKR}
    Let $E$ be a Morava $E$-theory of height $n$. There is a natural ring homomorphism
    \[
    \chi \colon \pi_0 E^{BG} \to \cl_n(G,C_0)
    \]
    for any finite group $G$, such that the induced map
    \[
    \chi \colon \pi_0 E^{BG} \otimes_{\pi_0(E)} C_0 \to \cl_n(G,C_0)
    \]
    is an isomorphism of rings.

    Via the natural $GL_n(\integers_p)$-action on $C_0$, the map
    \[
    \chi \colon \pi_0 E^{BG} \otimes \rationals \to \cl_n(G,C_0)
    \]
    lands in the fixed points and induces an isomorphism onto the subring of $GL_n(\integers_p)$-invariants.
\end{theorem}

The remainder of this subsection is devoted to recalling the definition of the character map~$\chi$.

To define the map
\[
\chi \colon \pi_0 E^{BG} \to \cl_n(G,C_0),
\]
it suffices to specify, for each conjugacy class $[f]$ with $f \in G_n$, a ring homomorphism
\[
\pi_0 E^{BG} \to C_0,
\]
since $\cl_n(G,C_0)$ is canonically a product of copies of $C_0$ indexed by the $G$-orbits in $G_n$.

Let $f \in \hom(\integers_p^n,G)$. The image
\[
\integers_p^n \to \im f
\]
is a finite abelian $p$-group. Although conjugation in $G$ changes the inclusion $\im f \subset G$, the induced maps
\[
E^{BG} \to E^{B\im f}
\]
are homotopic, since the corresponding maps of classifying spaces are homotopic. Thus, to define $\chi$, it suffices to construct a map
\[
\pi_0 E^{B\im f} \to C_0.
\]

\begin{proposition}
    The ring $C_0$ admits a canonical map from $\pi_0 E^{B(\integers/p^r)^n}$.
\end{proposition}

\begin{proof}
    The ring $\pi_0 E^{B(\integers/p^r)^n}$ classifies $n$-tuples of points of order dividing $p^r$ in the Quillen formal group $Q$. Using the canonical identification
    \[
    (\rationals_p/\integers_p)^n \cong Q \times \spf C_0,
    \]
    we obtain $n$ distinguished $p^r$-torsion points corresponding to the elements
    \[
    e_i^r = (0,\dots,0,1/p^r,0,\dots,0) \in (\rationals_p/\integers_p)^n.
    \]
    There is a unique ring homomorphism
    \[
    \pi_0 E^{B(\integers/p^r)^n} \to C_0
    \]
    classifying these points.
\end{proof}

Since $\im f$ is finite, there exists $r$ sufficiently large such that the map
\[
\integers_p^n \to \im f
\]
factors through $(\integers/p^r)^n$. This yields a morphism
\[
\pi_0 E^{B\im f} \to \pi_0 E^{B(\integers/p^r)^n}.
\]
We therefore obtain a composite
\[
\chi_{[f]} \colon
\pi_0 E^{BG}
\to
\pi_0 E^{B\im f}
\to
\pi_0 E^{B(\integers/p^r)^n}
\to
C_0.
\]

It remains to check that this construction is independent of the choice of $r$. For $r+1 \ge r$, the projection
\[
\integers/p^{r+1} \to \integers/p^r
\]
induces a map
\[
\pi_0 E^{B(\integers/p^r)^n} \to \pi_0 E^{B(\integers/p^{r+1})^n}.
\]
On points, this map sends a $p^{r+1}$-torsion point $x$ to the $p^r$-torsion point $[p](x)$. In particular, it sends each $e_i^{r+1}$ to $e_i^r$, so the above composite is independent of $r$.

\begin{definition}
    The HKR character map is the ring homomorphism
    \[
    \chi \colon \pi_0 E^{BG} \to \cl_n(G,C_0).
    \]
    On the factor corresponding to the conjugacy class of
    \[
    f \colon \integers_p^n \to G,
    \]
    it is given by the composite
    \[
    \chi_{[f]} \colon
    \pi_0 E^{BG}
    \to
    \pi_0 E^{B\im f}
    \to
    \pi_0 E^{B(\integers/p^r)^n}
    \to
    C_0.
    \]
\end{definition}

\subsection{The character of the total power operation}

In \cite{BS}, the authors use HKR character theory to construct a map that allows one to calculate the rationalization of the total power operation for the ring $\pi_0 E^{BG}$. We devote the following discussion to stating the definition of this map in the case $G = \explicitset{e}$. Namely, we recall the definition of a map
\[
\Psi \colon C_0 \rightarrow \cl_n(\Sigma_m, C_0),
\]
for which:

\begin{proposition}\label{prop:psidiagram}\cite[Theorem 9.1]{BS}
    The following diagram commutes:
    \[
        \begin{tikzcd}
            \pi_0 E \ar[r, "\widetilde{\psi}"] \ar[d, "\chi"]
            & \pi_0 E^{B\Sigma_m} \ar[d, "\chi"]\\
            C_0 \ar[r, "\Psi"] & \cl_n(\Sigma_m, C_0)
        \end{tikzcd}
    \]
\end{proposition}

\begin{remark}
    The formula we give for $\Psi$ here uses finite-index subgroups $H \subset \integers_p^n$, whereas the definition in \cite{BS} uses finite subgroups $H \subset (\rationals_p/\integers_p)^n$. These are in bijection with each other via the assignment sending a finite subgroup $H \subset (\rationals_p/\integers_p)^n$ to the image of the dual of the isogeny
    \[
        (\rationals_p/\integers_p)^n \longrightarrow (\rationals_p/\integers_p)^n/H .
    \]
\end{remark}

\begin{definition}
    Let
    \[
        \Sum_m = \explicitset{\sum H_i \mid H_i \leq \integers_p^n,\;\; \sum \cardinality{\integers_p^n : H_i} = m } .
    \]
    This is the set of formal sums of subgroups of $\integers_p^n$ whose indices sum to $m$.
\end{definition}

\begin{proposition}
    There is a natural isomorphism $(\Sigma_m)_n \rightarrow \Sum_m(\integers_p^n)$ in the sense of Definition~\ref{def:nclasses}.
\end{proposition}

\begin{proof}
    The map is defined as follows. A morphism $f \colon \integers_p^n \rightarrow \Sigma_m$ determines an action of $\integers_p^n$ on the set $[m] = \explicitset{1,\ldots,m}$. Denote the orbits of this action by $O_1,\ldots,O_k$. We send the class $[f]$ to the formal sum $\sum \stab O_i$ of the stabilizers of these orbits. By the orbit--stabilizer theorem, we have
    \[
        \sum \cardinality{\integers_p^n : \stab O_i} = m .
    \]

    It is a classical result in the theory of permutation groups that two elements are conjugate if and only if they have the same orbit sizes. To see that the map above is an isomorphism, we need the following generalization: two homomorphisms $\phi,\psi \colon G \rightarrow \Sigma_m$ are conjugate if and only if they define isomorphic $G$-sets of size $m$.

    If $\sigma \colon [m]_\phi \rightarrow [m]_\psi$ is an isomorphism of $G$-sets, in the sense that
    \[
        \sigma(\phi(g)(i)) = \psi(g)(\sigma(i))
    \]
    for all $g \in G$ and $i \in [m]$, then by definition $\phi$ and $\psi$ are conjugate by $\sigma$. The converse direction follows by reversing this argument.
\end{proof}

\begin{proposition}\cite[Proposition 4.1]{BS}\label{prop:actionofmn}
    The ring $C_0$ carries a natural left action of the monoid $M_n^{\geq 0}$ of all matrices in $M_n(\integers_p)$ with nonvanishing determinant, extending the action from Proposition~\ref{prop:actionongln}.
\end{proposition}

\begin{proof}[sketch of proof]
    Let $A \in M_n(\integers_p)$ have nonvanishing determinant. Then $A$ defines an isogeny
    \[
        [A] \colon (\rationals_p/\integers_p)^n \rightarrow (\rationals_p/\integers_p)^n .
    \]
    Let $M = \ker[A] \subset (\rationals_p/\integers_p)^n$ denote the kernel of the map induced by $A$ via its action on $(\rationals_p/\integers_p)^n$, obtained using the identification
    \[
        (\rationals_p/\integers_p)^n \cong \hom(\integers_p^n, \rationals_p/\integers_p) .
    \]

    Let $R$ be a complete local ring and let $p \colon C_0 \rightarrow R$ be a point classifying an isomorphism
    \[
        s \colon (\rationals_p/\integers_p)^n \xrightarrow{\sim} Q \times \spf R ,
    \]
    where $Q \times \spf R$ is a deformation. This induces an isomorphism
    \[
        (\rationals_p/\integers_p)^n/M \rightarrow (Q \times \spf R)/s(M) .
    \]

    There is a canonical isomorphism
    \[
        \psi^H \colon (\rationals_p/\integers_p)^n/M \rightarrow (\rationals_p/\integers_p)^n
    \]
    given by the first isomorphism theorem applied to the map $A$. Indeed, the short exact sequence
    \[
        0 \rightarrow H \rightarrow (\rationals_p/\integers_p)^n \xrightarrow{A} (\rationals_p/\integers_p)^n \rightarrow 0
    \]
    exhibits $(\rationals_p/\integers_p)^n$ as the quotient of $(\rationals_p/\integers_p)^n$ by $H$.

    We obtain a formal group $(Q \times \spf R)/s(M)$ over $R$ together with an isomorphism
    \[
        (\rationals_p/\integers_p)^n \rightarrow (Q \times \spf R)/s(M)
    \]
    given by the composition
    \[
        (\rationals_p/\integers_p)^n \rightarrow (\rationals_p/\integers_p)^n/M \rightarrow (Q \times \spf R)/s(M) .
    \]
    The formal group $(Q \times \spf R)/s(M)$ is again a deformation. Thus there is a corresponding $R$-point
    \[
        A^*p \colon C_0 \rightarrow R
    \]
    classifying $(Q \times \spf R)/s(M)$ together with the above isomorphism.

    Taking $R = C_0$, we obtain the desired action. Note that if $A$ is an isomorphism, then $M$ is trivial; in this case the action is given by composition with the canonical action on $(\rationals_p/\integers_p)^n$.
\end{proof}

For any subgroup $H \subset \integers_p^n$ of index $\leq p$, we choose a matrix $\phi_H \in M_n(\integers_p)$ with image $H$.

We can now give a formula for $\Psi$.

\begin{definition}\label{def:formulapsi}
    Given an element $a \in C_0$, set
    \[
        \Psi(a)\!\left(\sum H_i\right) = \prod (\phi_{H_i})_* a ,
    \]
    where $\phi_{H_i}$ is any matrix in $M_n^{\geq 0}$ with image $H_i$.
\end{definition}
\begin{remark}
    The map $\Psi$ depends strongly on the choice of the matrices $\phi_{H_i}$. However, Proposition~\ref{prop:psidiagram} holds for any such choice, essentially because the character lands in the fixed points of the $GL_n(\integers_p)$-action.
\end{remark}

In order to prove Theorem~\ref{thm:charactercomputation}, we need the following.

\begin{proposition}\label{prop:transideal}\cite[Proposition 6.2]{BS}
    The transfer ideal is $\pi_0 E^{B\Sigma_p} \cap J$, where $J$ is the factor of $\cl_n(\Sigma_p, C_0)$ corresponding to $\sum_{i=1}^p \integers_p^n$, a sum of $p$ copies of the whole group.
\end{proposition}

Now we can prove Theorem~\ref{thm:charactercomputation}.

\begin{proof}[Proof of Theorem~\ref{thm:charactercomputation}]
    By Theorem~\ref{thm:HKR}, we have a splitting
    \begin{equation}\label{eq:have}
        \pi_0 E^{B\Sigma_p} \otimes_{\pi_0 E} C_0
        \cong
        \pi_0 E \otimes_{\pi_0 E} C_0
        \times
        \parenth{\pi_0 E^{B\Sigma_p}/T_{\Sigma_p} \otimes_{\pi_0 E} C_0}.
    \end{equation}
    This comes from viewing the first factor $\pi_0 E \otimes_{\pi_0 E} C_0$ as $J$, and the second factor as
    $\pi_0 E^{B\Sigma_p}/T_{\Sigma_p} \otimes_{\pi_0 E} C_0$, which corresponds to sums involving proper subgroups.

    We first want to show that this is a base change of the desired splitting
    \begin{equation}\label{eq:want}
        \pi_0 E^{B\Sigma_p} \otimes \rationals
        \cong
        \pi_0 E \otimes \rationals
        \times
        \parenth{\pi_0 E^{B\Sigma_p}/T_{\Sigma_p} \otimes \rationals}.
    \end{equation}
    The left-hand side of~\eqref{eq:have} carries a $GL_n(\integers_p)$-action whose fixed points are precisely the left-hand side of~\eqref{eq:want}. To obtain the result, we first show that the splitting in~\eqref{eq:have} is respected by this action, namely that the action preserves each factor separately.

    Indeed, the action of $GL_n(\integers_p)$ on $(\Sigma_p)_n = \hom(\integers_p^n, \Sigma_p)$ fixes the trivial homomorphism, which corresponds to the first factor in the splitting of Theorem~\ref{thm:HKR}. Thus the action preserves $J$ as a subset. It follows that projection to the first factor on the right-hand side of~\eqref{eq:want} is given by evaluation at the trivial element of the formal group, while projection to the second factor is localization away from the trivial element.

    Furthermore, in the diagram of Proposition~\ref{prop:psidiagram}, the character maps $\chi$ land in the fixed points of the $GL_n(\integers_p)$-action by Theorem~\ref{thm:HKR}. We therefore obtain that the rationalization of $\widetilde{\psi^p}$ is given by the fixed points of $\Psi$, which is an equivariant map.

    On the second coordinate, we have the restriction of $\psi^p \otimes \rationals$ to
    $\pi_0 E^{B\Sigma_p}/T_{\Sigma_p} \otimes \rationals$, which agrees with $\psi^p$ by Proposition~\ref{prop:psidiagram}. On the first coordinate, we have the restriction of $\Psi$ to $J$, which by Definition~\ref{def:formulapsi} is given by
    \[
        x \longmapsto \prod_{i=1}^p (\phi_{\explicitset{e}})_* x = x^p .
    \]
\end{proof}

\section{Computation of $\delta$ at height $2$ and prime $3$}\label{sec:computation}

We have gathered all the tools needed to compute $\delta \colon \pi_0 E \rightarrow \pi_0 E$ at height $2$ and prime $3$. The following is a more precise version of Theorem~\ref{thm:computationofdelta}. In this section $E$ is morava $E$ theory of height $2$ over the prime $3$.

\begin{notation}\label{notation:ugly_stuff}
    \begin{enumerate}
        \item Let
    \[
    A=\begin{pmatrix}
        0 & 0 & 0 & 0 & 0 & 0 & 0 & 3 \\
        1 & 0 & 0 & 0 & 0 & 0 & 0 & 3 \, c \\
        0 & 1 & 0 & 0 & 0 & 0 & 0 & h - 9 \\
        0 & 0 & 1 & 0 & 0 & 0 & 0 & -9 \, c \\
        0 & 0 & 0 & 1 & 0 & 0 & 0 & -6 \, h + 12 \\
        0 & 0 & 0 & 0 & 1 & 0 & 0 & -{\left(h - 1\right)}c - 7 \, c \\
        0 & 0 & 0 & 0 & 0 & 1 & 0 & -3 \, h + 3 \\
        0 & 0 & 0 & 0 & 0 & 0 & 1 & -3 \, c
    \end{pmatrix},
    \]
    be the companion matrix of $f$.
    \item Let
    \begin{align*}
        B &=
        \frac{
            c(6 \, h^{2} - 111 \, h + 361 \, )\,A^7
            +(17 \, h^{3} - 334 \, h^{2} + 1377 \, h - 1124)\,A^6
        }{h-17}\\
        &+
        \frac{
            3\, c \parenth{5 \,  h^{3} - 99 \,  h^{2} + 416 \,  h - 354 \, } A^5
            +\parenth{3 \, h^{4} - 18 \, h^{3} - 566 \, h^{2} + 3134 \, h - 2617}A^4
        }{h-17}\\
        &+
        \frac{
            c\parenth{ h^{4} - 56 \, h^{3} + 855 \,  h^{2} - 4069 \,h + 5509 } A^3
            +\parenth{3 \, h^{4} - 114 \, h^{3} + 1323 \, h^{2} - 4564 \, h + 3032}A^2
        }{h-17}\\
        &+
        \frac{
            3 \, c\parenth{ h^{3} - 17 \,  h^{2} + 48 \,  h - 128} A
            +h^{4} - 44 \, h^{3} + 579 \, h^{2} - 2301 \, h + 1557
        }{h-17}.
    \end{align*}
    This is a polynomial in $A$ with coefficients in $\pi_0 E$, since $h-17$ is invertible in $\pi_0 E$.
    \item Let $x \in \pi_0 E = \integers_3[[h]][c]/(c^2+1-h)$ be regarded as a formal power series in the variable $h$, and let $x(B)$ denote the matrix-valued series obtained by substituting $h = B$.
    \end{enumerate}
\end{notation}

\begin{proposition}
    For any $x \in \pi_0 E = \integers_3[[h]][c]/(c^2+1-h)$, the series $x(B)$ converges to a matrix $x(B) \in M_8(\pi_0 E)$ in the $(3,h)$ adic topology.
\end{proposition}
\begin{proof}
    In order to show that $x(B)$ converges we must show that $B^n \xrightarrow{n\rightarrow \infty} 0$. A matrix's powers converge to $0$ in the adic topology if and only if, modulo each maximal ideal it is nilpotent.
    
    We first show that $A^n \xrightarrow{n\rightarrow \infty} 0$. Since $A$ is the companion matrix of $f$ and $f\equiv x^8\pmod{(3,h)}$ it is nilpotent modulo the maximal ideal. $B$ is nilpotent modulo the maximal ideal since the constant term of the polynomial from defining $B$ in terms of $A$ is \[\frac{h^{4} - 44 \, h^{3} + 579 \, h^{2} - 2301 \, h + 1557}{h-7} \equiv 0 \pmod{(3,h)}.\]
\end{proof}

The goal of this section is to prove the following Theorem.

\begin{theorem}\label{thm:computationofdelta}
    The operation $\delta$ is given by
    \[
        \delta(x) = 3\, x - \frac{x^3 + \tr x(B)}{3}.
    \]
\end{theorem}

Recall from Corollary~\ref{cor:complofmodcurves} that, in order to calculate the maps $\eta$ and $\psi^3$, one should complete a map of moduli spaces of elliptic curves with level structures. Using concrete models for $X$, $X_0(3)$, and $X_1(3)$, Yifei Zhu computed in \cite{ZHU} explicit formulas for $\pi_0 E^{B\Sigma_3}/T_{\Sigma_3}$, $\pi_0 E^{BC_3}/T_{C_3}$, $\psi^3$, and $\eta$ in the case $p=3$ and height $2$.

\begin{theorem}\label{thm:mainzhu}
In the notation of Proposition~\ref{prop:compofE0}:
\begin{enumerate}
    \item \cite[Remark~8]{ZHU} The ring $\pi_0 E^{B\Sigma_3}/T_{\Sigma_3}$ is given by $\pi_0 E\squarebrack{a}/W(a)$, where
    \[
        W(a) = a^4 - 6a^2 + (h-9)a - 3 .
    \]

    \item \cite[Corollary~9]{ZHU} The additive total power operation $\psi^3$ is given by
    \begin{align*}
        \psi^3(h)
        &= h^3 + \parenth{a^3-6a-27}h^3 + 3\parenth{-6a^3+a^2+36a+67}h + 57a^3-27a^2 - 334a - 342,\\
        \psi^3(i)&=-i.
    \end{align*}

    \item \cite[Proposition~4]{ZHU}, together with Corollary~\ref{cor:complofmodcurves}. The ring $\pi_0 E^{BC_3}/T_{C_3}$ is given by $\pi_0 E\squarebrack{u}/f(u)$, where
    \[
        f(u) = u^8 + 3cu^7 + 3(h - 1)u^6 + (h - 1)cu^5 + 7cu^5 + 6(h - 1)u^4 - 6u^4 + 9cu^3 - cu^2 + 8u^2 - 3cu - 3,
    \]
    and $c$ is the square root of $h-1$ in $\pi_0 E$ which is congruent to $i$ modulo $(h,3)$.

    \item Corollary~\ref{cor:complofmodcurves} and \cite[Proposition~6(ii)]{ZHU}. The ring homomorphism
    \[
        \eta \colon \pi_0 E^{B\Sigma_3}/T_{\Sigma_3} \rightarrow \pi_0 E^{BC_3}/T_{C_3}
    \]
    induced by the natural inclusion $C_3 \rightarrow \Sigma_3$ is given by
    \begin{align*}
        a \longmapsto
        & \frac{1}{17-h}\left(
        cu^7 + (3c^2-2)u^6 + (3c^3-6)u^5 + (c^4+c^2+2)u^4 \right. \\& \quad \left. 
        + (4c^3-15c)u^3 + (c^2+2)u^2 - 12cu - 18 \right).
    \end{align*}
\end{enumerate}
\end{theorem}

We have a commutative diagram
\[\begin{tikzcd}
	{\pi_0E} & {\pi_0E^{BC_p}} & {\pi_0E^{BC_p}/\text{transfer}} & {\mathcal{O}(X_1(p))}_{(p,h)}^{\wedge} \\
	{\pi_0E} & {\pi_0E^{B\Sigma_p}} & {\pi_0E^{B\Sigma_p}/\text{transfer}} & {\mathcal{O}(X_0(p))_{(p,h)}^{\wedge}}
	\arrow[from=2-1, to=1-1, equal]
	\arrow[from=2-1, to=2-2]
	\arrow[from=1-1, to=1-2, "\alpha"]
	\arrow[from=2-2, to=1-2]
	\arrow[from=2-4, to=1-4]
	\arrow[from=2-3, to=2-4, "\sim"]
	\arrow[from=1-3, to=1-4, "\sim"]
	\arrow[from=2-3, to=1-3]
	\arrow[from=1-2, to=1-3]
	\arrow[from=2-2, to=2-3]
\end{tikzcd}\]

\begin{proposition}\label{prop:compofetapsi}
    Using the notation of Theorem~\ref{thm:mainzhu}, the composition along the first row of the diagram is $\eta \circ \psi^3$. It is the map $\pi_0E \rightarrow \pi_0E[u]/(f(u))$ given by
    \begin{align*}
        h &\mapsto 
        \frac{
            c(6 \, h^{2} - 111 \, h + 361 \, )\,u^7
            +(17 \, h^{3} - 334 \, h^{2} + 1377 \, h - 1124)\,u^6
        }{h-17}\\
        &+
        \frac{
            3\, c \parenth{5 \,  h^{3} - 99 \,  h^{2} + 416 \,  h - 354 \, } u^5
            +\parenth{3 \, h^{4} - 18 \, h^{3} - 566 \, h^{2} + 3134 \, h - 2617}u^4
        }{h-17}\\
        &+
        \frac{
            c\parenth{ h^{4} - 56 \, h^{3} + 855 \,  h^{2} - 4069 \,h + 5509 } u^3
            +\parenth{3 \, h^{4} - 114 \, h^{3} + 1323 \, h^{2} - 4564 \, h + 3032}u^2
        }{h-17}\\
        &+
        \frac{
            3 \, c\parenth{ h^{3} - 17 \,  h^{2} + 48 \,  h - 128} u
            +h^{4} - 44 \, h^{3} + 579 \, h^{2} - 2301 \, h + 1557
        }{h-17}
    \end{align*}
\end{proposition}
\begin{proof}
    We computed $\eta \circ \psi^3$ as a polynomial in $u$ using Sage, and then used the Euclidean algorithm to reduce modulo $f(u)$ and obtain the stated formula.
\end{proof}

We will now compute $\alpha$.

\begin{proposition}\label{prop:computationofalpha}
    Let $A,B$ be as in Notation~\ref{notation:ugly_stuff}.
    Then $\alpha$ is given by $\alpha(x) = \frac{x^3 + \tr x(B)}{3}$.
\end{proposition}

\begin{proof}
    By Proposition~\ref{prop:alphaistraceoftotalpo}, we have that $\alpha$ sends an element to the trace of its image under the total power operation, divided by $3$. By Proposition~\ref{prop:rezkstrickland}, item~\ref{prop:rezkstrickland:goodgroups}, the map $\pi_0E^{BC_p} \rightarrow \pi_0E^{BC_p} \otimes \rationals$ is injective, so in order to compute the trace we may pass to the rationalization. By Theorem~\ref{thm:charactercomputation}, the trace of $\widetilde{\psi}(x)$ is $\tr(\eta\circ\psi(x)) + x^3$.

    By Cayley--Hamilton, $f(A) = 0$. Moreover, for $p(u)\in \pi_0E[u]/(f(u))$ we have $\tr p = \tr p(A)$, since the representing matrix for $u$ in the basis $1,u,u^2,\ldots,u^{8}$ is $A$. Thus $x(B)$ is the representing matrix of $\eta\circ\psi(x)$ in $\pi_0E^{BC_p}/T_{C_3}$ with respect to the same basis.
\end{proof}

\begin{proof}[Proof of Theorem~\ref{thm:computationofdelta}]
    Lemma~\cite[5.3.4]{CSY1} gives that $\cardinality{BC_p} = 3$.
    The theorem is now an immediate consequence of Proposition~\ref{prop:computationofalpha} and the definition of $\delta$.
\end{proof}

\end{document}